\documentclass[makeidx]{article}
\usepackage[latin1]{inputenc}
\usepackage{amsmath}
\usepackage{amsfonts}
\usepackage{amssymb}
\usepackage{graphicx}
\usepackage{hyperref}
\usepackage{amssymb,amsmath,amsthm,graphicx, float}

\title{Heat content and exit time moments under domain deformations}
\title{On the heat content functional and its critical domains}
\author{Alessandro Savo}

\date{}
\usepackage{amssymb , amsmath, amsthm}
\newtheorem{defi}{Definition} 
\newtheorem{thm}[defi]{Theorem}
\newtheorem{theorem}[defi]{Theorem}

\newtheorem{rem}[defi]{Remark}
 \newtheorem{prop}[defi]{Proposition}
\newtheorem{lemme}[defi]{Lemma}

\newtheorem{cor}[defi]{Corollary}

 \newcommand{\twosystem}[2]{\left\{\begin{aligned} &#1\\ &#2\end{aligned}\right.}
\newcommand{\threesystem}[3]{\left\{ \begin{aligned}&#1\\ &#2\\&#3\end{aligned}\right.}

\newcommand{\qtq}{\quad\text{on}\quad}

\newcommand{\nero}{\smallskip$\bullet\quad$\rm}

\newcommand{\parte}[1]{\smallskip\noindent {\rm#1)}\,\,}
\newcommand{\scal}[2]{\langle{#1},{#2}\rangle}

\newcommand{\abs}[1]{\lvert{#1}\rvert}

\newcommand{\reals}{{\bf R}}

\newcommand{\sphere}[1]{{\bf S}^{#1}}
\newcommand{\real}[1]{{\bf R}^{#1}}

\newcommand{\bd}{\partial}

\newcommand{\eps}{\epsilon}

\newcommand{\derive}[2]{\dfrac{\bd #1}{\bd#2}}

\newcommand{\matrice}{\begin{pmatrix}}
\newcommand{\ok}{\end{pmatrix}}

\def\<{\left\langle}
\def\>{\right\rangle}

\begin{document}

\maketitle

 

\abstract{We study and classify smooth bounded domains in an analytic Riemannian manifold which are critical for the heat content at all times $t>0$. We do that by first computing the first variation of the heat content, and then showing that $\Omega$ is critical if and only if it has the so-called constant flow property, so that we can use a previous  classification result established in \cite{S1} and \cite{S2}. The outcome is that $\Omega$ is critical for the heat content at time $t$, for all $t>0$, if and only if $\Omega$ admits an isoparametric foliation, that is, a foliation whose leaves are all parallel to the boundary and have constant mean curvature. Then, we consider the sequence of functionals given by the exit-time moments $T_1(\Omega),T_2(\Omega),\dots$, which generalize the torsional rigidity $T_1$.
We prove that $\Omega$ is critical for all $T_k$ if and only if $\Omega$ is critical for the heat content at every time $t$, and then we get a classification as well. 
The main purpose of the paper is to understand the variational properties of general isoparametric foliations and their role in PDE's theory; in some respects  they generalize the properties of the foliation of $\real n$ by Euclidean spheres. }

\bigskip

\noindent {\it Classification AMS $2000$}: 58J50, 35N30, 35P15, 58C40\newline
\newline
{\it Acknowledgments:} Research partially supported by INDAM and GNSAGA of Italy

 

\maketitle

\normalsize 
\parindent0cm

\section{Introduction and main results} In this paper we consider smooth bounded domains $\Omega$ in a complete Riemannian manifold $M$, and we study those domains which are critical for the heat content $H_{\Omega}(t)$ of $\Omega$ at any time fixed $t$, under deformations which preserve the inner volume (see the definition in  Section \ref{heatcontent}). The heat content has been amply studied in the literature and there are many important contributions. Of special relevance to this paper are the following aspects:  its asymptotic behaviour for small times (see \cite{vdBG1}, \cite{vdBG2}, \cite{S4}, \cite{S3}, \cite{Gil1}, \cite{Gil2}), comparison theorems of isoperimetric nature (\cite{BS}), its relation with the Dirichlet spectrum and the Brownian motion (\cite{mcDM}, \cite{mcD}), and the geometric rigidity of an associated overdetermined problem (\cite{S1}, \cite{S2}).

\smallskip

 The main scope of this paper is to understand the geometry of critical domains, and possibly give a classification. We observe that that it is easy to prove a classification when $\Omega$ is a domain in Euclidean space $\real n$, hyperbolic space ${\bf H}^n$ or the hemisphere $\sphere n_+$. In all these cases the only critical domains are geodesic balls, which in fact  are absolute maxima for the heat content, by known comparison results (\cite{BS}).
 
 \smallskip
 
 So, our main focus will be on arbitrary ambient Riemannian manifolds $M$. After computing the formula for the first variation of the heat content, we see that critical domains are characterized by the so-called {\it constant flow property} (defined in section \ref{constantfp}), which basically says that the normal derivative of the temperature function (the {\it heat flow}) is constant on the boundary of the domain.  The overdetermined heat equation which follows has been studied in the papers \cite{S1}, \cite{S2} and there, under the assumption that the ambient manifold is analytic, it is shown that a domain supports a solution to such overdetermined heat equation if and only if it is an {\it isoparametric tube}, that is, it admits a foliation by smooth, parallel, constant mean curvature hypersurfaces, degenerating to a leaf of possibly lower dimension; this singular leaf is always a smooth, closed, minimal submanifold of $M$ (see Section \ref{isotubes}).
 
 \smallskip

The fact that the only (compact) isoparametric tubes of $\real n, {\bf H}^n$ and 
 $\sphere n_+$ are the geodesic balls (in which case the leaves are geodesic spheres, and the singular leaf is simply a point, the center of the ball) is a consequence of a classical result due to Cartan; it also follows from the classical Alexandrov theorem and from its generalization to ${\bf H}^n$ and 
 $\sphere n_+$ (see \cite{MoRos}), to the effect that the only embedded, constant mean curvature hypersurfaces of those spaces are geodesic spheres. 
 That is why the full geometric meaning of criticality is perhaps not seen on those spaces. 
 
 \smallskip
 
 It is already seen, however, on the whole sphere $\sphere n$, which hosts a reasonably large, and quite interesting, variety of isoparametric tubes which are not geodesic balls: the simplest such family is given by domains bounded by Clifford tori (products of spheres of appropriate radii). More generally,  it follows from the classical theory that isoparametric tubes in the sphere are bounded by  isoparametric hypersurfaces, that is,  hypersurfaces having constant principal curvatures (we will be more precise in Theorem \ref{inthesphere} below).  For general facts about isoparametric hypersurfaces, see for example \cite{Tho}. The classification of isoparametric hypersurfaces in the sphere is a classical problem in differential geometry which started from Cartan in the 30's (\cite{Car}) and which, after several important intermediate results,  has been completed only recently (\cite{Ch}).

\smallskip
  
 Thus, as a result of our findings and thanks to the collective efforts leading to the final classification \cite{Ch}, the family of critical domains in the whole sphere can be  completely classified, and is in fact much larger than the family of geodesic balls. Criticality for the heat content, in $\sphere n$, implies in particular that any component of the boundary is a smooth algebraic variety, more precisely, it is the zero set of the restriction to $\sphere n$ of a harmonic polynomial in $\real{n+1}$ satisfying precise algebraic conditions ({\it Cartan-M\"untzner polynomials}). We will briefly summarize the main facts about isoparametric foliations on the sphere in section \ref{inthesphere}. 
 
\smallskip

At the moment, there is no classification (in the strict sense of the word) of isoparametric foliations in general, although there is some progress. It is clear however that this condition is quite strict, and that a random manifold will not support any isoparametric foliation. 

\smallskip

The probabilistic aspect of the heat content, its relation with  the Dirichlet spectrum and with the so-called {\it exit-time moments} $\{T_1(\Omega), T_2(\Omega),\dots\}$ have  been studied in several papers (\cite{CLmcD}, \cite{DLmcD}, \cite{mcDM}, \cite{mcD}, \cite{HMP2}, \cite{HMP1}).  We point out that $T_1(\Omega)$ has also an intepretation in mechanics, being the {\it torsional rigidity} of the domain $\Omega$; isoperimetric inequalities for the functional $T_1(\Omega)$ are classical, and were recently extended to the higher exit time moments. For example, it is shown in \cite{mcD} that, in the spaces $\real n, {\bf H}^n$ and $\sphere n_+$ the only domains which are critical for $T_k(\Omega)$, {\it for at least one} $k\in {\bf N}$, are geodesic balls.  We point out that the same rigidity definitely  does not extend to spherical domains not contained in a hemisphere; in fact all isoparametric tubes (and not only geodesic balls) are critical for $T_k(\Omega)$ {\it for all} $k\in{\bf N}$. So, the problem we want to bring to attention is the following: 

\nero classify the domains which are critical for all exit-time moments $T_k(\Omega)$ when $\Omega$ is a domain in an  arbitrary Riemannian manifold. 

\smallskip

The outcome is that this family coincides with the family of domains having the constant flow property, hence also with the family of domains which are critical for the heat content functional at all times: these domains are precisely the isoparametric tubes.  At this point, we summarize the main results and state the main theorem of this paper. In what follows, "critical" means "critical under volume preserving deformations".

\begin{thm}\label{main} Let $\Omega$ be a smooth bounded domain in a complete Riemannian manifold $(M,g)$. The following statements are equivalent:

\smallskip

\item a) $\Omega$ is critical for the heat content $H_{\Omega}(t)$ at every fixed time $t>0$.

\item b) $\Omega$ is critical for the $k$-th exit time moment $T_k(\Omega)$, for all $k\geq 1$.

\item c) $\Omega$ has the constant flow property.

\smallskip

If the ambient Riemannian manifold $M$ is analytic, a), b) and c) are in turn all equivalent to:

\smallskip

\item d) $\Omega$ is an isoparametric tube over a closed, minimal submanifold of $M$.
\end{thm}

This paper is a natural continuation of  \cite{S1} and \cite{S2}, and clarifies the variational meaning of the constant flow property. Its main theme is to stress the relation between variational geometry, overdetermined PDE's and the isoparametric property, which is not fully evident when studying problems in the spaces $\real n, {\bf H}^n$ and $\sphere n_+$ due to the Alexandrov rigidity theorem.  That the isoparametric condition  is sufficient to guarantee the existence of solutions  to many overdetermined problems was perhaps first discussed in detail by Shklover in \cite{Shk}. That this property is actually also necessary (for the constant flow property) was the main outcome of \cite{S1} and \cite{S2}.

The main theorem  shows that certain variational problems, and the overdetermined PDE's they generate, naturally lead to the consideration of isoparametric foliations.  For previous facts on PDE's and the isoparametric theory, we mention the works \cite{Sol1} and \cite{Sol2}, 
where the spectrum of the Laplace-Beltrami operator was computed for cubic isoparametric minimal hypersurfaces of the Euclidean sphere, and the paper \cite{TY} where the first eigenvalue of any closed minimal isoparametric hypersurface of ${\bf S}^{n+1}$ was shown to be equal to $n$, thus confirming Yau's conjecture in these cases.

\smallskip

Another important functional is the Dirichlet heat trace, which is known to be a spectral invariant of a domain $\Omega$. For results about its critical domains, we refer to \cite{ESI}: there it is shown, using the asymptotic expansion of the heat trace and the Alexandrov theorem, that the only critical domains in $\real n, {\bf H}^n$ and $\sphere n_+$ are the geodesic balls. At the moment, there seems to be no classification of domains which are critical for the heat trace in other manifolds, in particular, in the sphere $\sphere n$ of arbitrary  dimension.

\smallskip

In the rest of this introduction, we give the precise definitions and discuss the problem in greater detail. Proofs will be given starting from Section 2, where we prove the first variation formula of the heat content, Theorem \ref{vformula}, which is the main step.
  

\subsection{The heat content} \label{heatcontent} Let $\Omega$ be a smooth, bounded domain in a complete Riemannian manifold $M$. The heat content of $\Omega$ is the function $H_{\Omega}(t)$ of time $t>0$ which measures the total heat inside $\Omega$ at time $t$, assuming that initially the temperature is uniformly distributed and equal to $1$, and that the boundary is subject to absolute refrigeration at all times (Dirichlet boundary conditions). The heat content is expressed as
$$
H_{\Omega}(t)=\int_{\Omega}u(t,x)dv(x)
$$
where $u(t,x)$ is the temperature function, solution of the heat equation:\begin{equation}\label{temp}
\threesystem
{\Delta u+\derive{u}{t}=0\quad\text{on}\quad\text{on}\quad (0,\infty)\times\Omega}
{u(0,x)=1\quad\text{for all }\quad x\in\Omega}
{u(t,y)=0\quad\text{for all }\quad y\in\bd\Omega, \quad t>0}
\end{equation}

It is natural to expect that the extrinsic geometry of $\bd\Omega$  plays a major role in the asymptotic behavior of $H_{\Omega}(t)$ for {\it small times} $t$; in fact, for small times,  only the points near the boundary will feel the sudden drop of the temperature due to the boundary refrigeration (in fact, points far away from the boundary obey the so-called {\it principle of not feeling the boundary}).  This behavior is reflected in an asymptotic series in powers of $\sqrt t$:
\begin{equation}\label{asyseries}
H_{\Omega}(t)\sim \abs{\Omega}+\sum_{k=1}^{\infty}\beta_k(\Omega)t^{\frac k2}\quad\text{as}\quad t\to 0
\end{equation}
whose coefficients are, for what we have just said,  supported on the boundary and depend on the extrinsic curvature of $\bd\Omega$, precisely, the second fundamental form and its covariant derivatives, which in turn involve the ambient curvature tensor together with its covariant derivatives. After some works in Euclidean space, the existence of the asymptotic series in the general Riemannian case, together with the calculation of the coefficients $\beta_k(\Omega)$ up to $k=2$ was carried out by van den Berg and Gilkey \cite{vdBG1}. For more general operators of Laplace-type, with various boundary conditions, see the review papers \cite{Gil1} and  \cite{Gil2}. An algorithm for the computation of the whole asymptotic series \eqref{asyseries} has been developped by the author in \cite{S4} and \cite{S3}; this approach has been recently adapted in \cite{RiRo2} to treat the sub-Riemannian case.

On the other hand, for {\it large times} the heat content is governed by the Dirichlet spectrum of the domain; in particular, $H_{\Omega}(t)$ decays exponentially to zero as $t\to\infty$, with speed proportional to the lowest Dirichlet eigenvalue $\lambda_1(\Omega)$:
\begin{equation}\label{largetimes}
H_{\Omega}(t)\sim c_{\Omega}^2e^{-\lambda_1(\Omega)t}\quad\text{as $t\to\infty$}
\end{equation}
with $c_{\Omega}$ being the integral over $\Omega$ of a unit $L^2$-norm first Dirichlet eigenfunction. This can be easily seen by writing the Fourier series of the heat content. 
\subsection{Constant flow property}\label{constantfp}

The function $\derive{u}{N}(t,\cdot):\bd\Omega\to\reals$ is called the (pointwise) {\it heat flow} at time $t$; when integrated on the boundary, it measures the speed at which the domain is loosing heat  due to boundary refrigeration:
$$
\dfrac{d}{dt}H_{\Omega}(t)=-\int_{\bd\Omega}\derive{u}{N}(t,y)d\sigma(y)
$$
where $d\sigma$ is the Riemannian surface measure.

Now, for a general domain the heat flow is not uniform (i.e. constant) across the boundary. 
If the heat flow is constant on $\bd\Omega$, at every fixed value of time, then we say that $\Omega$ has the {\it constant flow property}. In other words:

\smallskip

\noindent {\bf Definition.}  {\it The domain $\Omega$ has the {\rm constant flow property} if there exists a smooth function $c: (0,\infty)\to\reals$ such that
\begin{equation}\label{cfp}
\derive{u}{N}(t,y)=c(t)
\end{equation}
for all $y\in\bd\Omega$.}

Adding the condition \eqref{cfp} to the heat equation \eqref{temp} defining $u(t,x)$, we obtain an overdetermined problem, and this overdetermination imposes strict conditions on the geometry of the domain. In the papers \cite{S1} and \cite{S2} we actually give a geometric characterization of the domains with the constant flow property when the ambient manifold $M$ is analytic: they are  {\it isoparametric tubes} (see the definition below) that is, they admit an isoparametric foliation or, in other words, a foliation by parallel hypersurfaces having constant mean curvature. We will discuss this in more details in the next section. 

Before doing that, we remark that for domains in $\real n$ another interesting overdetermined condition on the function $u(t,x)$ was studied in \cite{MagSak} and the works that followed (in particular, the recent work \cite{Sak}): namely, the existence of a stationary isothermic hypersurface $S$ in $\Omega$. Under mild assumptions, it is proved in \cite{MagSak} there that existence of such $S$ forces $\Omega$ to  be a ball. 


\subsection{Isoparametric tubes}\label{isotubes} Here is the definition.

\begin{defi}\label{iso}  Let $P$ be a smooth, closed submanifold of the domain $\Omega$ of dimension $n$. We say that $\Omega$ is a {\rm smooth tube around $P$} if :
\item a) $\Omega$ is the set of points at distance at most $R$ from $P$,

\item b) For each $s\in (0,R]$, the equidistant
$$
\Sigma_s=\{x\in \Omega: d(x,P)=s\}
$$
is a smooth hypersurface of $\Omega$. 

\smallskip

We say that the smooth tube $\Omega$ is an {\rm isoparametric tube} if every equidistant $\Sigma_s$ as above has constant mean curvature.
\end{defi}

The submanifold $P$ is called the {\it soul} of $\Omega$, and can have dimension 
$\dim P=0,\dots,n-1$. The soul is then  an embedded submanifold, which is always {\it minimal} (see \cite{GeTa}).   Any isoparametric tube has at most two boundary components (for the easy proof see \cite{S2}).

For example, a solid revolution  torus in $\real 3$ with radii $a>b>0$ is a smooth tube (the soul $P$ is a circle), but is not an isoparametric tube because equidistants have variable mean curvature. In fact:

\begin{prop}\label{spaceform} The only (compact) isoparametric tubes in the spaces $\real n, {\bf H}^n$ and $\sphere n_+$ are the geodesic balls (in which case the soul is a point).
\end{prop}

For the proof, just observe that the Alexandrov theorem holds true in any of these spaces, so that  any compact embedded hypersurface must be a geodesic sphere. If the boundary has two components, then by definition they must be spheres with the same (constant) mean curvature (hence the same radius) and must be parallel, which is impossible in the cases at hand.  Thus, the boundary consists of one piece, which is then a sphere.

 \nero Observe that a domain in a {\it whole} sphere $\sphere n$, bounded by two geodesic spheres, is an isoparametric tube if and only if the two boundary spheres are isometric and have equal (or antipodal) centers : in that case, the soul is an equatorial (i.e. totally geodesic) hypersurface. In fact, if the centers are neither equal nor antipodal then the region is not even a smooth tube. 

\nero More generally, any geodesic ball in a locally harmonic manifold is (more or less by definition) an isoparametric tube around its center.

\nero Now assume that the metric of $\Omega$ is smooth and that there is a distinguished point $p\in\Omega$ such that $(\Omega\setminus\{p\},g)$ is isometric to
$(0,R]\times\sphere{n-1}$ endowed with the metric $g=dr^2+\theta^{2}(r)g_{\sphere{n-1}}$ ($r$ being the radial parameter). Then we say that $\Omega$ is a {\it revolution manifold}. Clearly any revolution manifold  is an isoparametric tube around its soul, the point $p$. Note that its boundary $\bd\Omega$ has only one component, namely $\{R\}\times\sphere{n-1}$.  

\medskip

In the next subsection we will discuss the main class of examples of isoparametric tubes, namely, spherical domains bounded by isoparametric hypersurfaces. 


\subsection{Isoparametric tubes in the standard sphere}\label{inthesphere} 

Usually, a hypersurface $\Sigma$ of a Riemannian manifold $M$ is called {\it isoparametric} if all nearby parallel hypersurfaces have constant mean curvature. 
Thus, the boundary of an isoparametric tube, and all of its regular equidistants, are isoparametric hypersurfaces of $M$, by definition. 

\smallskip

Now it is well-known that a hypersurface $\Sigma$ of a space form $\real n,{\bf H}^n$ or $\sphere n$, is {\it isoparametric} if and only if it has constant principal curvatures, that is, if and only if the characteristic polynomial of its second fundamental form is the same at all points of $\Sigma$. 

While on $\real n,{\bf H}^n$ and the hemisphere $\sphere n_+$ the only compact isoparametric hypersurfaces are the geodesic spheres, in  $\sphere n$ there is a much larger variety of them.  The classification of such hypersurfaces started from Cartan and was a major problem in Differential Geometry, which came to a complete solution only very recently (\cite{Ch}). Let us review the main properties of an isoparametric hypersurface $\Sigma$ of $\sphere n$.

\nero The number $g$ of distinct principal curvatures of $\Sigma$ can be only $1,2,3,4$ and $6$ (\cite{Mu1}).  The case $g=1$ corresponds to the family of geodesic spheres, and $g=2$ corresponds to Clifford tori; these are tubes around a totally geodesic submanifold and are hypersurfaces of type:
$$
\Sigma=\sphere p(a)\times\sphere q(b), \quad p+q=n-1, \quad a^2+b^2=1
$$
which admit a natural embedding into $\sphere{n}$ with constant principal curvatures $\lambda=\frac ba$ (counted $p$ times) and $\mu=-\frac ab$ (counted $q$ times).

It is a remarkable and perhaps surprising fact that when $g=4$ there exists non-homogenous isoparametric hypersurfaces. 

\nero For each $\Sigma$ there exist two regular, connected submanifolds $\Sigma_+, \Sigma_-$ of $\sphere n$ such that $\Sigma$ is the surface of the tube with radius $r_+$ (resp.  $r_-$) around $\Sigma_+$ (resp.  $\Sigma_-$). These submanifolds are called the {\it focal submanifolds} of $\Sigma$, and are minimal in $\sphere n$. 

\nero Every isoparametric hypersurface $\Sigma$ belongs to a one-parameter family, which gives rise to what is known to be an {\it isoparametric foliation} of $\sphere n$. This foliation has precisely two singular leaves (the focal submanifolds $\Sigma_+$ and $\Sigma_-$) and contains exactly one minimal isoparametric hypersurface: when $g=1$ it is the unique equator of the family (which is totally geodesic), and when $g=2$, for fixed $p$ and $q$, it is the minimal Clifford torus defined by the identity  $pb^2=qa^2$. 

\medskip

{\bf Geometric properties.} 
In what follows, $\Sigma$ is an isoparametric hypersurface of $\sphere n$.
We recall here the main geometric facts due to M\"unzner (see \cite{Mu1} and \cite{Mu2}); for further details we refer to the papers of Cecil  \cite{Cec} and Shklover (\cite{Shk}, page 17).
List the distinct principal curvatures of $\Sigma$ in decreasing order, as follows:
$$
k_1>k_2>\dots>k_g,
$$
so that (here and below) $g$ denotes the number of distinct principal curvatures. 
It turns out that $k_i=\cot\theta_i$ for a sequence $0<\theta_1<\dots<\theta_g<\pi$ such that
$$
\theta_j=\theta_1+\frac{j-1}{g}\pi.
$$
Let $m_j$ be the multiplicity of $k_j$. Then, we have a cyclic behavior: $m_{j+2}=m_j$ (modulo $g$) which implies that the sequence of multiplicities $m_1, m_2,\dots$ is determined by $m_1$ and $m_2$; in particular, $m_1=m_2=\dots=m_g$ whenever $g$ is odd.  Set
\begin{equation}\label{defc}
c=\frac 12(m_2-m_1)g^2.
\end{equation}
Note that $c=0$ if and only if all multiplicities are equal; this holds whenever $g$ is odd and also when $g=6$, by a result of M\"unzner. When $g=2$ one has $c=0$ if and only if $n$ is odd and $\Sigma$ is a Clifford torus
$\sphere p(a)\times\sphere p(b)$, that is, $p=q$.

\smallskip

It turns out that $\Sigma$ is always a (smooth) real algebraic variety: in fact, $\Sigma$ is a regular level set of the restriction to $\sphere{n}$ of a homogeneous polynomial $F:\real{n+1}\to\reals$ of degree $g$ which satisfies the conditions
\begin{equation}\label{cmp}
\twosystem
{\abs{\bar\nabla F}^2=g^2\abs{x}^{2g-2}}
{\bar\Delta F=c\abs{x}^{g-2}}
\end{equation}
where $c$ is as in \eqref{defc}  and $\bar\Delta, \bar\nabla$ are the Laplacian and the gradient in $\real{n+1}$. A polynomial $F$ with the properties \eqref{cmp} is an example of {\it Cartan-M\"unzner polynomial}. Conversely, any Cartan-M\"unzner polynomial of degree $g$ with a constant $c\ne \pm (n-1)g$ gives rise to an isoparametric foliation of $\sphere n$ with $g$ distinct principal curvatures. Having that, the geometric classification of isoparametric foliations reduces to the (difficult) algebraic problem of classifying 
all Cartan-M\"unzner polynomials.

\smallskip

We can now classify all isoparametric tubes in $\sphere n$ (for the proof of Theorem \ref{inthesphere} see the Appendix).

\begin{thm}\label{inthesphere} Let $\Omega$ be a domain in $\sphere n$. Then $\Omega$ is an isoparametric tube if and only if :

\smallskip

a) either $\Omega$ is the domain bounded by a connected isoparametric hypersurface, 

\smallskip

b) or $\Omega$ is a tube around a minimal isoparametric hypersurface $\Sigma$ such that all its distinct principal curvatures have the same multiplicity (that is, $\Sigma$ is minimal with $c=0$).
\end{thm} 

In the first case the boundary is connected and the soul is a focal submanifold of $\bd\Omega$; in the second case the soul is $\Sigma$ and the boundary consists of two parallel isoparametric hypersurfaces, which are at the same distance to  $\Sigma$ and have the same mean curvature (with respect to the inner normal vector). 

\smallskip

 For what we have just said we see that, in low dimensions:

\begin{cor}\label{lowd}

\parte a An isoparametric tube in $\sphere 2$  is either a geodesic disk or a tube around an equator.

\parte b An isoparametric tube in $\sphere 3$  is congruent to one of the following: a geodesic ball, a tube around the equator, a domain bounded by a Clifford torus or a tube around the minimal Clifford torus $\sphere 1(\frac{1}{\sqrt 2})\times \sphere 1(\frac{1}{\sqrt 2})$.
\end{cor}


\subsection{Exit time moments}\label{exittimem} The function $E_1:\Omega\to\reals$, unique solution of the Dirichlet problem
$$
\twosystem
{\Delta E_1=1\qtq \Omega}
{E_1=0\qtq \bd\Omega}
$$
is known in the literature as the {\it torsion function} of the domain $\Omega$, and its integral
$$
T_1(\Omega)\doteq\int_{\Omega}E_1\,dv_g
$$
defines the so-called {\it torsional rigidity} of $\Omega$. Isoperimetric inequalities for the torsion function and the torsional rigidity are by now classical (see for example \cite{Ba}, \cite{PoSz}). 
It is also well-known (but see below) that a domain is critical for torsional rigidity (under volume preserving deformations) if and only if its torsion function has constant normal derivative, that is, if and only if  $\Omega$ supports a solution to the overdetermined problem
\begin{equation}\label{serrin}
\twosystem
{\Delta u=1\qtq\Omega}
{u=0, \quad\derive uN=c \qtq\bd\Omega}
\end{equation}
known in the literature as {\it Serrin problem}. As a consequence of Serrin's rigidity theorem (see \cite{Ser}) we know that the only domains in $\real n$ which are critical for $T_1$ under volume preserving deformations are balls, and these are all maxima by the classical isoperimetric result by Polya (\cite{Po}, \cite{PoSz}). The same rigidity holds in the spaces ${\bf H}^n$ and $\sphere n_+$: the only critical domains are geodesic balls (also in this case, these are absolute maxima, see for example \cite{BS} and \cite{CGL}). 

On the sphere there are many critical domains for torsional rigidity which are not geodesic balls, for example, domains bounded by any connected isoparametric hypersurface, as proved in \cite{Shk} (see also \cite{S1}).  Yet more generally, isoparametric tubes in general Riemannian manifolds  are critical for $T_1$ (\cite{S2}).
However, the condition of being critical for $T_1$ seems to be too weak to guarantee a reasonable classification: see a recent example in \cite{FMW} of a spherical domain which is critical for $T_1$ and is not even an isoparametric tube (actually, the boundary has variable  mean curvature). 
We will prove in this paper a classification result under criticality for the whole family of exit time moments, which we are going to define.  

\smallskip

Now, the function $E_1$ has also a probabilistic intepretation, being the {\it mean exit time} associated to the Brownian motion of $\Omega$. As such, it is part of a hierarchy of exit time moments.Precisely, define the {\it $k$-th exit time function} $E_k$ on $\Omega$ inductively as follows. We set $E_0=1$ and, for $k\geq 1$, we let $E_k$ be the unique solution of
\begin{equation}\label{esubk}
\twosystem
{\Delta E_k=kE_{k-1}}
{E_k=0\quad\text{on}\quad\bd\Omega}
\end{equation}
Note that $E_1$ is just the torsion function. 
The {\it $k$-th exit time moment} of the bounded domain $\Omega$ is now defined as
\begin{equation}\label{tsubk}
T_k(\Omega)\doteq\int_{\Omega}E_k dv.
\end{equation}
The sequence
$$
m^{\star}(\Omega)=\{T_1(\Omega), T_2(\Omega), \dots\}
$$
is known as the {\it exit time moment spectrum} of $\Omega$. 
These invariants have been studied in the papers \cite{CLmcD}, \cite{DLmcD},
\cite{HMP1},\cite{HMP2},  \cite{mcD}, \cite{mcDM}.

\smallskip

The following expressions hold, for $k\geq 1$ (see for example \cite{CLmcD}):
\begin{equation}\label{expression}
E_k(x)=k\int_0^{\infty}t^{k-1}u(t,x)\,dt
\end{equation}
therefore
\begin{equation}\label{expressiontk}
T_k(\Omega)=k\int_0^{\infty}t^{k-1}H_{\Omega}(t)\,dt,
\end{equation}
where $H_{\Omega}(t)$ is the heat content of $\Omega$. These identities show the strict relation between the exit time moments and the heat content. 
The relation between the moment spectrum and the Dirichlet spectrum has been studied in \cite{DLmcD}, \cite{CLmcD} and \cite{HMP2}.

In this paper we characterize the geometry of Riemannian domains which are critical for all exit time moments $T_k(\Omega)$, see Theorem \ref{mainbis} below.


\subsection{Domain deformations and critical domains}\label{deformations} 
Let $V$ be a smooth vector field defined in a neighborhood $U$ of the domain $\Omega$ in $M$. Define a one-parameter domain deformation 
$
f_{\eps}:\Omega\to M
$
by:
\begin{equation}\label{deformation}
f_{\eps}(x)=\exp_x(\eps V(x)).
\end{equation}
For $\eps$ small enough, $f_{\eps}$ restricts to a diffeomorphism:
$$
f_{\eps}:\Omega\to f_{\eps}(\Omega)\doteq\Omega_{\eps}.
$$
We call $\Omega_{\eps}$ the {\it one-parameter deformation of $\Omega$ associated to the vector field $V$}.

\smallskip

Given a geometric functional $\mathcal F=\mathcal F(\Omega)$ depending smoothly on the domain $\Omega$ we define its first variation $\mathcal F'(\Omega)$ as follows:
$$
\mathcal F'(\Omega)=\dfrac{d}{d\eps}|_{\eps=0}\mathcal F(\Omega_{\eps}).
$$
Note that $\mathcal F'(\Omega)$ depends on the vector field $V$ which defines the deformation, hence it would be more correct to write
$
\mathcal F'(\Omega)=D\mathcal F(\Omega,V),
$
interpreting such expression as the directional derivative of the functional $\mathcal F$ at $\Omega$ in the direction $V$. 
We will say that $\Omega$ is {\it critical for the functional $\mathcal F$} if 
$$
D\mathcal F(\Omega,V)=0
$$
for all deformations of $\Omega$ hence, for all vector fields $V$. However, to have a meaningful geometric problem one should  (and we will) impose  that the deformation $f_{\epsilon}$ is {\it volume preserving}:
$$
\abs{\Omega_{\eps}}=\abs{\Omega}
$$ 
for $\epsilon$ small enough.  To preserve volume, the vector field $V$ must satisfy the condition $
\int_{\Omega}{\rm div}V=0
$
and then, by Green formula:
\begin{equation}\label{vp}
\int_{\bd\Omega}\scal{V}{N}=0.
\end{equation}
Under condition \eqref{vp} we can study the first variation of the heat content when the deformed domains have the same volume.


\subsection{The main result} Here is the main result of this paper, as stated at the beginning of the introduction. The word {\it critical} means really {\it critical under volume preserving deformations}.

\begin{thm}\label{mainbis} Let $\Omega$ be a smooth bounded domain in a complete Riemannian manifold $(M,g)$. The following are equivalent:

\smallskip

\item a) $\Omega$ is critical for the heat content $H_{\Omega}(t)$ at every fixed time $t>0$.

\item b) $\Omega$ is critical for the $k$-th exit time moment $T_k(\Omega)$, for all $k\geq 1$.

\item c) $\Omega$ has the constant flow property.

\smallskip

If the ambient Riemannian manifold $M$ is (real) analytic, then a), b) and c) are in turn all equivalent to:

\smallskip

\item d) $\Omega$ is an isoparametric tube over a closed, minimal submanifold of $M$.
\end{thm}


\subsection{Remarks} 

We first recall a comparison result proved by Burchard and Shmuckenschlager \cite{BS}:  let $\Omega$ be a domain in a constant curvature space form (hence, up to homotheties, $\real n, {\bf H}^n$ and $\sphere n$)  and let $\Omega^{\star}$ be the geodesic ball with the same volume: $\abs{\Omega^{\star}}=\abs{\Omega}$.  Then, at all times $t>0$ one has
$$
H_{\Omega}(t)\leq H_{\Omega^{\star}}(t).
$$
Therefore geodesic balls in constant curvature space forms are absolute maxima for the heat content functional, if one restricts to deformations which preserve the inner volume. 

\smallskip

We have seen in Theorem \ref{spaceform} that the only isoparametric tubes in  $\real n, {\bf H}^n$ and $\sphere n_+$ are geodesic balls; hence, by our main theorem, we have:

\begin{cor} The only bounded domains in $\real n, {\bf H}^n$ and $\sphere n_+$ which are critical for the heat content are geodesic balls (these are absolute maxima by \cite{BS}).
\end{cor}

The same conclusion holds for the exit time moments; this was first proved in \cite{mcD} by the Alexandrov-Serrin moving plane argument. We remark that the moving plane method cannot be applied in our general case, and in fact
the situation changes drastically already in the whole sphere, as there are many critical domains which are not geodesic balls. This follows from part d) of the main theorem and the classification of the isoparametric tubes 
given in Theorem \ref{inthesphere} of the previous section:

\begin{cor} A domain $\Omega$ in $\sphere n$ is critical for the heat content at all times $t$ and for the $k$-th exit time moment $T_k$, for all $k$,  if and only if: $\Omega$ is bounded by a (connected) isoparametric hypersurface or is a tube around a minimal isoparametric hypersurface having $c=0$.
\end{cor}

For example, in $\sphere 3$, the critical domains are: geodesic balls, tubes around an equator,  domains bounded by a Clifford torus and tubes around a minimal Clifford torus. The classification exists in higher dimension, but it gets more complicated due to the large variety of isoparametric foliations in higher dimensions.

\nero As a final remark, we ask the following question: is it possible to weaken the assumption {\it "$\Omega$ is critical for all $T_k$"} in the statement of the main theorem ? For example, if we assume criticality for only one $k$ can we get the required rigidity? Well, the answer is negative, at least if $k=1$:

\begin{rem} There exist (analytic) Riemannian domains which are critical for torsional rigidity $T_1$ but are not isoparametric tubes. 
\end{rem}
A first such example was constructed in \cite{S2}, and consists of any minimal free boundary immersion in the $3$-dimensional unit Euclidean ball $B_3$ having more than two boundary components: the normal derivative of the torsion function $E_1$ is constant on the boundary, but $\Omega$ cannot be a smooth tube.
Other "exotic" examples exist even in the round sphere; in \cite{FMW} one can find domains which are critical for torsional rigidity but have boundary with non-constant mean curvature. These domains are in fact suitable perturbations of a tubular neighborhood of the equator.


\subsection{First variation of the heat content}\label{firstvariation} The main step in the proof of Theorem \ref{main} is the formula for the first variation of the heat content, which is the following.

\begin{theorem}\label{vformula} For a fixed value of time $t>0$, let $\mathcal F_t(\Omega)=H_{\Omega}(t)$ be the heat content of $\Omega$ at time $t$. Then its first variation in the direction $V$ is given by:
$$
D\mathcal F_t(\Omega,V)=-\int_0^t\int_{\bd\Omega}\scal{V}{N}\derive{u}{N}(\tau,y)\derive{u}{N}(t-\tau,y)\,dv(y)\,d\tau
$$
where $N$ is the inner unit normal and $u(t,x)$ is the temperature function defined in \eqref{temp}.
\end{theorem}

Note that, as $t\to 0$, one has, for all $y\in\bd\Omega$ (see \cite{S3}):
$$
\derive{u}{N}(t,y)=\dfrac{1}{\sqrt\pi}\dfrac{1}{\sqrt t}+O(1)
$$
which guarantees that the integral on the right hand side is convergent. We remark that Ozawa obtained in \cite{Oza} the first variation of the Dirichlet heat kernel of Euclidean domains. Our methods are different and are based on the fact that the heat kernel depends smoothly on the deformation parameter (see \cite{RS}). 



\subsection{Scheme of proof} The rest of the paper is organized as follows.

In Section 2 we prove the first variation formula, Theorem \ref{vformula}. In Section 3 we prove 
 the equivalence between a) and c) of the main theorem:
 
\begin{thm}\label{equivone} $\Omega$ is critical for the functional $\mathcal F_t(\Omega)$ given by the heat content at time $t$, for all $t>0$, if and only if $\Omega$ has the constant flow property.
\end{thm}

In Section 3 we show the equivalence between b) and c):

\begin{theorem}\label{etm} The domain $\Omega$ is critical for the $k$-th exit time moment, for all $k\geq 1$ and for all volume preserving deformations, if and only if $\Omega$ has the constant flow property.
\end{theorem}

Having done that, we finish the proof of the main theorem by recalling that,
when $(M,g)$ is analytic, the equivalence between c) and d) has been proved in Theorem 7 of \cite{S2}. With this in mind, the proof of Theorem \ref{mainbis} is  complete.


\section{Proof of the first variation formula}

 Let $\Omega_{\eps}=f_{\eps}(\Omega)$ be a smooth deformation of $\Omega$ associated to the vector field $V$, as in \eqref{deformation}.  We adopt the following point of view: deforming the domain in a manifold with a fixed metric is equivalent to keeping the domain fixed and deforming the metric. In more precise terms,  we identify $(\Omega_{\eps},g)$ with $(\Omega,f_{\eps}^{\star}g)$ and denote $g_{\eps}\doteq f_{\eps}^{\star}g$. We let $\Delta_{\eps}$ be the Laplacian associated to the metric $g_{\eps}$ of $\Omega$, and $dv_{\eps}$ the corresponding Riemannian measure.  We denote:
$$
g'\doteq\dfrac{d}{d\eps}|_{\eps=0}g_{\eps},\quad \Delta'\doteq\dfrac{d}{d\eps}|_{\eps=0}\Delta_{\eps}, \quad dv'=\dfrac{d}{d\eps}|_{\eps=0}dv_{\eps}.
$$
It is a standard fact that:
\begin{equation}\label{standard}
dv'={\rm div}V\,dv=-\delta V\,dv,
\end{equation}
where $dv=dv_g$ is the Riemannian measure associated to the original metric $g$ and $\delta$ is the adjoint of the gradient operator.  We point out that the operator $\Delta'$ has been computed by Berger in \cite{B}; the explicit expressions of $g'$ and $\Delta'$  will be given in Lemma \ref{prime} of Appendix 1. 

Let us denote $u_{\eps}(t,x)$ the temperature function in the deformed metric $g_{\eps}$. It is the unique solution of 
$$
\threesystem
{\Delta_{\eps} u_{\eps}+\derive{u_{\eps}}{t}=0\quad\text{on}\quad (0,\infty)\times\Omega}
{u_{\eps}(0,x)=1\quad\text{for all }\quad x\in\Omega}
{u_{\eps}(t,y)=0\quad\text{for all }\quad y\in\bd\Omega, \quad t>0}
$$
and one has 
\begin{equation}\label{hk}
u_{\eps}(t,x)=\int_{\Omega}k_{\eps}(t,x,y)dv_{\eps}(y)\quad\text{and of course}\quad u(t,x)=\int_{\Omega}k(t,x,y)dv(y)
\end{equation}
where $k_{\eps}(t,x,y)$ is the Dirichlet heat kernel of $(\Omega,g_{\eps})$ (resp. the Dirichlet heat kernel of $(\Omega,g)$).
From Proposition 6.1 in \cite{RS}, we know that the map $\eps\mapsto k_{\eps}$ is differentiable \footnote{Looking at the proof, we see that when $k_{\epsilon}$ is the Dirichlet heat kernel on functions (i.e. forms of degree zero) the assumption made in \cite{RS} that the normal direction is the same for all deformed metrics is actually not needed for differentiability.}. Therefore, from  the expression \eqref{hk} also the map $\eps\mapsto u_{\eps}(t,x)$ is differentiable and we denote:
$$
u'(t,x)=\dfrac{d}{d\eps}|_{\eps=0}u_{\eps}(t,x).
$$
Accordingly, we denote
\begin{equation}\label{hepsilon}
H_{\Omega_{\eps}}(t)=\int_{\Omega}u_{\eps}(t,x)\,dv_{\eps}(x), 
\quad H'_{\Omega}(t)=\dfrac{d}{d\eps}|_{\eps=0}H_{\Omega_{\eps}}(t)
\end{equation}
(we stress that the prime indicates differentation with respect to $\epsilon$ along $V$, and not with respect to time $t$, which in this discussion is fixed).

\smallskip

The proof goes as follows: we first express the function $u'$ in terms of the Dirichlet heat kernel of $g$ and $\Delta'u$. Then, 
using an explicit expression of the operator $\Delta'$, and integration by parts, we obtain the final statement. To avoid discussing, at every step, the convergence of the integrals involved we use a suitable approximation $v(t,x)$ of the function $u(t,x)$ by a small parameter $\delta$ and then pass to the limit as $\delta\to 0$ to obtain the statement. 

\subsection{Approximation of $u(t,x)$} We fix $\delta>0$ and small and we introduce the function $v:(0,\infty)\times \Omega\to\reals$ defined by:
$$
v(t,x)=u(t+\delta,x)
$$
(for simplicity of notation, we omit to write explicitly the dependance of $v$ on $\delta$). 

Then, $v$ satisfies the heat equation on $(\Omega,g)$ with initial data $v(0,x)=u(\delta,x)$ and Dirichlet boundary conditions. Since the initial data vanishes on the boundary, $v(t,x)$ extends to a smooth function on $[0,\infty)\times\bar\Omega$ and we thus can avoid dealing with the distributional behavior of $u(t,\cdot)$ near the boundary, for small times, in the sense that any derivative of $v$ is uniformly bounded on $[0,\infty)\times\bar\Omega$. Since:
$$
\derive ut(t,x)=-\int_{\bd\Omega}\derive{k}{N_y}(t,x,y)\,d\sigma(y)<0
$$
we see that $u(t,x)$ is decreasing in $t$ at any point $x$: this implies that
$v(t,x)\leq u(t,x)\leq 1$ on $(0,\infty)\times\Omega$, and so, as both functions vanish on the boundary:
\begin{equation}\label{normalder}
\derive vN(t,y)\leq \derive uN(t,y)
\end{equation}
for all $t>0$ and $y\in\bd\Omega$. Finally, it is clear that $v(t,\cdot)\to u(t,\cdot)$ together with all of its derivatives, as $\delta\to 0$.

\smallskip

Let $(\Omega,g_{\eps})$ be a smooth deformation of $(\Omega,g)$ and let $v_{\eps}(t,x)$ be the solution of the Dirichlet heat equation in $(\Omega, g_{\eps})$ with initial data $u(\delta,x)$ (not depending on $\eps$). Hence $v_{\eps}$ satisfies:
$$
\threesystem
{\Delta_{\eps} v_{\eps}+\derive{v_{\eps}}{t}=0\quad\text{on}\quad (0,\infty)\times\Omega}
{v_{\eps}(0,x)=u(\delta,x)\quad\text{for all }\quad x\in\Omega}
{v_{\eps}(t,y)=0\quad\text{for all }\quad y\in\bd\Omega, \quad t>0}
$$ 
We introduce the following notation for a fixed $\delta>0$:
$$
v'(t,x)=\dfrac{d}{d\eps}|_{\eps=0}v_{\eps}(t,x), \quad
H^{(\eps)}_{\Omega,\delta}(t)=\int_{\Omega}v_{\eps}(t,x)dv_{\eps}(x), \quad H'_{\Omega,\delta}(t)=\dfrac{d}{d\eps}|_{\eps=0}H^{(\eps)}_{\Omega,\delta}(t).
$$

\begin{lemme} For any fixed $t>0$ one has:
$
u'(t,x)=\lim_{\delta\to 0}v'(t,x)
$
and therefore:
$$
D\mathcal F_t(\Omega,V)\doteq H'_{\Omega}(t)=\lim_{\delta\to 0}H'_{\Omega,\delta}(t).
$$
\end{lemme}

\begin{proof} First observe that
$$
v_{\eps}(t,x)=\int_{\Omega}k_{\eps}(t,x,y)u(\delta,y)\,dv_{\eps}(y)
$$
hence, differentiating both sides with respect to $\epsilon$ and setting $\eps=0$:
$$
v'(t,x)=\int_{\Omega}k'(t,x,y)u(\delta,y)\,dv(y)+\int_{\Omega}k(t,x,y)u(\delta,y)dv'(y).
$$
We now let $\delta\to 0$ on both sides; as $u(\delta,x)$ is uniformly bounded by $1$ for all $\delta$ and converges to $1$  as  $\delta\to 0$ we get, for all fixed $t>0$:
$$
\lim_{\delta\to 0}v'(t,x)=\int_{\Omega}k'(t,x,y)\,dv(y)+\int_{\Omega}k(t,x,y)dv'(y)=u'(t,x)
$$
which is the first assertion. The second assertion follows from the first by a similar argument, passing to the limit in:
\begin{equation}\label{hprimedelta}
H'_{\Omega,\delta}(t)=\int_{\Omega}v'(t,x)dv(x)+\int_{\Omega}v(t,x)dv'(x).
\end{equation}
\end{proof}

\subsection{Duhamel principle} We now fix $\delta$ and work on a convenient expression of $H'_{\Omega,\delta}(t)$.  Derive both sides of the equation
$
\Delta_{\eps}v_{\eps}+\derive{v_{\eps}}{t}=0
$
and set $\eps=0$ to get:
$$
\Delta'v+\Delta v'+\derive{v'}{t}=0.
$$
Therefore, since the initial data of $v_{\eps}$ does not depend on $\eps$, the function $v'(t,x)$ satisfies the following heat equation in the original metric $g$:
$$
\threesystem
{\Delta v'(t,x)+\derive{v'}{t}(t,x)=-\Delta'v(t,x) \quad\text{for all}\quad (t,x)\in(0,\infty)\times\Omega}
{v'(0,x)=0\quad\text{for all}\quad x\in\Omega}
{v'(t,y)=0\quad\text{for all}\quad y\in\bd\Omega, \,t>0}
$$
so that, by Duhamel principle:
$$
v'(t,x)=-\int_0^t\int_{\Omega}k(t-\tau,x,y)\Delta'v(\tau,y)\,dv(y)d\tau.
$$
Integrating on $\Omega$ with respect to $x$ we see (use Fubini and observe $\lim_{\delta\to 0}v(t,x)=u(t,x)$):
$$
\begin{aligned}
\int_{\Omega}v'(t,x)dv(x)&=-\int_0^t\int_{\Omega}u(t-\tau,y)\Delta'v(\tau,y)\,dv(y)d\tau\\
&=-\lim_{\delta\to 0}\int_0^t\int_{\Omega}v(t-\tau,y)\Delta'v(\tau,y)\,dv(y)\,d\tau\\
&=-\frac 12\lim_{\delta\to 0}\int_0^t\int_{\Omega}\Big(v(t-\tau,y)\Delta'v(\tau,y)+v(\tau,y)\Delta' v(t-\tau,y)\Big)\,dv(y)\,d\tau
\end{aligned}
$$
Using \eqref{hprimedelta}, knowing that $\lim_{\delta\to 0}v(t,x)=u(t,x)$ and that $dv'=-\delta V\,dv$ we arrive at the following expression.

\begin{lemme}\label{hprime} In the above notation:
$$
H'_{\Omega}(t)=-\frac 12\lim_{\delta\to 0}\int_0^t\int_{\Omega}\Big(v(t-\tau,y)\Delta'v(\tau,y)+v(\tau,y)\Delta' v(t-\tau,y)\Big)\,dv(y)\,d\tau-\int_{\Omega}u(t,x)\delta V(x)dv(x).
$$
\end{lemme}

The final step is to deal with the integral involving the operator $\Delta'$. This is done in the following lemma.

\begin{lemme}\label{rellich} If the functions $f,h$ vanish on the boundary one has:
$$
\begin{aligned}
\int_{\Omega}(f\Delta'h+h\Delta' f)dv=&
-2\int_{\Omega}\Big(\scal{V}{\nabla f}\Delta h+\scal{V}{\nabla h}\Delta f\Big)dv+
\int_{\Omega}(f\Delta h+h\Delta f)\delta Vdv\\
&+2\int_{\bd\Omega}\scal{V}{N}\derive fN\derive hNd\sigma
\end{aligned}
$$
\end{lemme}

\begin{proof} See Appendix 1. 
\end{proof}


\subsection{Proof of Theorem \ref{vformula}} 
We use Lemma \ref{hprime} and apply Lemma \ref{rellich} taking $f=v({\tau},\cdot)$ and $h=v({t-\tau},\cdot)$. We obtain
\begin{equation}\label{terma}
\begin{aligned}
-\frac 12\int_0^t\int_{\Omega}&\Big(v(t-\tau,y)\Delta'v(\tau,y)+v(\tau,y)\Delta' v(t-\tau,y)\Big)\,dv\,d\tau\\
&=\int_0^t\int_{\Omega}\Big(\scal{V(y)}{\nabla v({\tau},y)}\Delta v({t-\tau},y)+\scal{V(y)}{\nabla v({t-\tau},y)}\Delta v(\tau,y)\Big)\,dv\,d\tau\\
&-\frac 12\int_0^t\int_{\Omega}\Big(v(\tau,y)\Delta v(t-\tau,y)+v(t-\tau,y)\Delta v(\tau,y)\Big)\delta V(y)dv\,d\tau\\
&-\int_0^t\int_{\bd\Omega}\scal{V}{N}(y)\derive {v}N(t-\tau,y)\derive {v}N(\tau,y)\,d\sigma \,d\tau
\end{aligned}
\end{equation}
(integration on $\Omega$ is of course taken with respect to the variable $y$).

\smallskip

The first term. Exchange order of integration, and recall that $v$ satisfies the heat equation, so that $\Delta v(t-\tau,y)=-\derive{v}{t}(t-\tau,y)=\derive{v}{\tau}(t-\tau,y)$. Integrate by parts in the integral involving $\tau$. Get:
\begin{equation}\label{termaa}
\begin{aligned}
\int_0^t&\int_{\Omega}\Big(\scal{V(y)}{\nabla v({\tau},y)}\Delta v({t-\tau},y)dvd\tau\\
&=\int_{\Omega}\int_0^t\derive v{\tau}(t-\tau,y)\scal{V(y)}{\nabla v(\tau,y)}\,d\tau dv\\
&=\int_{\Omega}\Big[v(t-\tau,y)\scal{V(y)}{\nabla v(\tau,y)}\Big]_{\tau=0}^{\tau=t}dv-\int_{\Omega}\int_0^tv(t-\tau,y)\scal{V(y)}{\nabla\derive vt(\tau,y)}\,d\tau dv\\
&=\int_{\Omega}\Big(v(0,y)\scal{V(y)}{\nabla v(t,y)}-v(t,y)\scal{V(y)}{\nabla v(0,y)}\Big)\,dv(x)-\int_{\Omega}\int_0^tv(t-\tau,y)\scal{V(y)}{\nabla\derive vt(\tau,y)}\,d\tau dv\\
\end{aligned}
\end{equation}

By Green formula, since $v(t,\cdot)$ vanishes on the boundary:
\begin{equation}\label{inter}
\begin{aligned}
\int_{\Omega}v(0,y)\scal{V(y))}{\nabla v(t,y)}dv&=\int_{\Omega}v(t,y)\delta(v(0,y)V)dv\\
&=-\int_{\Omega}v(t,y)\scal{V(y)}{\nabla v(0,y)}dv+\int_{\Omega}v(t,y)v(0,y)\delta V(y)\,dv
\end{aligned}
\end{equation}

Inserting \eqref{inter} in \eqref{termaa} we conclude:
\begin{equation}\label{addone}
\begin{aligned}
\int_0^t\int_{\Omega}\scal{V(y)}{\nabla v({\tau},y)}\Delta v({t-\tau},y)dvd\tau
&=\int_{\Omega}\Big(-2v(t,y)\scal{V(y)}{\nabla v(0,y)}+v(0,y)v(t,y)\delta V(y)\Big)dv\\
&-\int_{\Omega}\int_0^tv(t-\tau,y)\scal{V(y)}{\nabla\derive{v}{t}(\tau,y)}d\tau dv
\end{aligned}
\end{equation}

On the other hand, using Green formula, since $v(t,\cdot)$ vanishes on the boundary:
\begin{equation}\label{addtwo}
\begin{aligned}
\int_0^t&\int_{\Omega}\scal{V(y)}{\nabla v(t-\tau,y)}\Delta v(\tau,y)dv\,d\tau=
\int_0^t\int_{\Omega}\scal{(\Delta v(\tau,y)) V(y)}{\nabla v(t-\tau,y)}dv\,d\tau\\
&=\int_0^t\int_{\Omega}v(t-\tau,y)\delta\big(\Delta v(\tau,\cdot)V\big)dv\,d\tau\\
&=-\int_0^t\int_{\Omega}v(t-\tau,y)\scal{\nabla\Delta v(\tau,y)}{V(y)}dv\,d\tau
+\int_0^t\int_{\Omega}v(t-\tau,y)\Delta v(\tau,y)\delta V(y)\,dvd\tau\\
&=\int_0^t\int_{\Omega}v(t-\tau,y)\scal{\nabla\derive{v}{t}(\tau,y)}{V(y)}dv\,d\tau
+\frac 12\int_0^t\int_{\Omega}
\Big(v(t-\tau,y)\Delta v(\tau,y)+v(\tau,y)\Delta v(t-\tau,y)\Big)\delta V(y)\,dvd\tau
\end{aligned}
\end{equation}
Adding \eqref{addone} and \eqref{addtwo} we arrive at:
\begin{equation}\label{addthree}
\begin{aligned}
\int_0^t\int_{\Omega}&\Big(\scal{V(y)}{\nabla v(\tau,y)}\Delta v(t-\tau,y)+\scal{V(y)}{\nabla v(t-\tau,y)}\Delta v(\tau,y)\Big)\,dv\,d\tau\\
=&\int_{\Omega}\Big(-2v(t,y)\scal{V(y)}{\nabla v(0,y)}+v(0,y)v(t,y)\delta V(y)\Big)dv\\
&+\frac 12\int_0^t\int_{\Omega}
\Big(v(t-\tau,y)\Delta v(\tau,y)+v(\tau,y)\Delta v(t-\tau,y)\Big)\delta V(y)dvd\tau
\end{aligned}
\end{equation}
Substituting \eqref{addthree} in \eqref{terma} we obtain:
\begin{equation}\label{total}
\begin{aligned}
-\frac 12\int_0^t\int_{\Omega}&\Big(v(t-\tau,y)\Delta'v(\tau,y)+v(\tau,y)\Delta' v(t-\tau,y)\Big)\,dv\,d\tau\\
&=\int_{\Omega}\Big(-2v(t,y)\scal{V(y)}{\nabla v(0,y)}+v(0,y)v(t,y)\delta V(y)\Big)dv\\
&-\int_0^t\int_{\bd\Omega}\scal{V}{N}(y)\derive {v}N(t-\tau,y)\derive {v}N(\tau,y)\,d\sigma\,d\tau
\end{aligned}
\end{equation}

We now pass to the limit as $\delta\to 0$ in \eqref{total}. Recall that  $v(t,\cdot)\to u(t,\cdot)$ and $v(0,y)\to 1$; by Green formula, since $u(t,\cdot)$ vanishes on $\bd\Omega$, we see that
$$
\lim_{\delta\to 0}\int_{\Omega}(-2v(t,y)\scal{V(y)}{\nabla v(0,y)})\,dv=
-2\int_{\Omega}\delta(u(t,\cdot)V)\,dv=-2\int_{\bd\Omega}u(t,\cdot)\scal{V}{N}\,d\sigma=0
$$
On the other hand, for all $s\in (0,t)$ one has that  $\derive{v}{N}(s,y)$ converges to $\derive{u}{N}(s,y)$ (from below, see \eqref{normalder}). Therefore, by the dominated convergence theorem, we conclude:
$$
\begin{aligned}
-\frac 12\lim_{\delta\to 0}&\int_0^t\int_{\Omega}\Big(v(t-\tau,y)\Delta'v(\tau,y)+v(\tau,y)\Delta' v(t-\tau,y)\Big)\,dv(y)\,d\tau\\
&=\int_{\Omega}u(t,y)\delta V(y)\,dv-\int_0^t\int_{\bd\Omega}\scal{V}{N}(y)\derive {u}N(t-\tau,y)\derive {u}N(\tau,y)\,d\sigma(y)\,d\tau
\end{aligned}
$$
and taking into account Lemma \ref{hprime} we arrive at the desired formula:
$$
H'_{\Omega}(t)=-\int_0^t\int_{\bd\Omega}\scal{V}{N}(y)\derive {u}N(t-\tau,y)\derive {u}N(\tau,y)\,d\sigma(y)\,d\tau.
$$


\section{Equivalence between parts a) and c) of Theorem \ref{main}} 

In this section we prove the following fact.

\begin{thm} $\Omega$ is critical for the functional $\mathcal F_t(\Omega)$ given by the heat content at time $t$, for all $t>0$, if and only if $\Omega$ has the constant flow property.
\end{thm}

Recall the formula for the first variation (Theorem \ref{vformula}). 
$$
D\mathcal F_t(\Omega,V)=-\int_0^t\int_{\bd\Omega}\scal{V}{N}\derive{u}{N}(\tau,y)\derive{u}{N}(t-\tau,y)\,d\sigma(y)\,d\tau.
$$
First assume that $\Omega$ has the constant flow property. Then, by assumption, $\derive{u}{N}(t,y)=c(t)$ for all $y\in\bd\Omega$ and for a smooth function $c=c(t)$.  Then:
$$
D\mathcal F_t(\Omega,V)=-\int_0^tc(\tau)c(t-\tau)\,d\tau\cdot\int_{\bd\Omega}\scal{V}{N}d\sigma=0,
$$
because the deformation  preserves the volume (see \eqref{vp}).  Hence $\Omega$ is critical. 

\smallskip

Conversely, assume that $D\mathcal F_t(\Omega,V)=0$ for all deformation vector fields $V$ and for all $t>0$, so that, interchanging the order of integration, we have by assumption:
\begin{equation}\label{zeromean}
\int_{\bd\Omega}\scal{V}{N}(y)B(t,y)\,d\sigma(y)=0
\end{equation}
where we set:
\begin{equation}\label{phi}
B(t,y)\doteq \int_0^t\derive{u}{N}(\tau,y)\derive{u}{N}(t-\tau,y)d\tau.
\end{equation}
As  $\scal{V}{N}$ has zero mean on $\bd\Omega$, and is otherwise arbitrary, we see that \eqref{zeromean} forces $B(t,y)$ to be constant in $y$, that is:
\begin{equation}\label{constant}
B(t,y)=\gamma(t)
\end{equation}
for a smooth (positive) function $\gamma:(0,\infty)\to\reals$. Let  the hat denote Laplace tranform with respect to time $t\in (0,\infty)$. Applied to $u(t,x)$ we obtain
the function 
\begin{equation}\label{uhat}
\hat u(s,x)=\int_0^{\infty}e^{-st}u(t,x)\,dt.
\end{equation}
We can differentiate in the normal direction under the integral sign in \eqref{uhat}, so that, for all $y\in\bd\Omega$ and $s>0$:
\begin{equation}\label{laplacet}
\derive{\hat u}{N}(s,y)=\int_0^{\infty}e^{-st}\derive{u}{N}(t,y)\,dt=\Big(\widehat{\derive uN}\Big)(s,y),
\end{equation}
where on the right we have the Laplace trasform of the heat flow.  We see, from \eqref{phi}, \eqref{constant} and the usual convolution rule:
$$
\widehat\gamma(s)=\hat B(s,y)=\Big(\widehat{\derive uN}\Big)^2(s,y)
$$
This means that
$
\Big(\widehat{\derive uN}\Big)(s,y)=\sqrt{\widehat\gamma(s)}
$
depends only on $s$. Pick two points $y_1,y_2\in\bd\Omega$. From what we have just said and \eqref{laplacet} we have, for all $s>0$:
$$
0=\Big(\widehat{\derive uN}\Big)(s,y_1)-\Big(\widehat{\derive uN}\Big)(s,y_2)=\int_0^{\infty}e^{-st}\Big(\derive{u}{N}(t,y_1)-\derive{u}{N}(t,y_2)\Big)\,dt
$$
and by the injectivity of Laplace transform:
$$
\derive{u}{N}(t,y_1)=\derive{u}{N}(t,y_2)
$$
for all $t>0$ and $y_1,y_2\in\bd\Omega$. Then, the function $\derive uN(t,\cdot)$ is constant on the boundary, and $\Omega$ has the constant flow property.


\section{Equivalence between parts b) and c) of Theorem \ref{main}} 

In this section we prove:

\begin{theorem}\label{etm} The domain $\Omega$ is critical for the $k$-th exit time moment, for all $k\geq 1$ and for all volume preserving deformations, if and only if $\Omega$ has the constant flow property.
\end{theorem} 

The proof uses the calculation of the first variation of the $k$-th exit time moment $T_k(\Omega)$ first done \cite{mcD}. For convenience of the reader we derive the same formula, using our approach,  in  Appendix 2 in this paper.

\begin{lemme}\label{eight} The first variation of the $k$-th exit time moment $T_k$, at $\Omega$, in the direction $V$, is given by:
\begin{equation}\label{fvetm}
DT_k(\Omega,V)=-\int_{\bd\Omega}\scal{V}{N}\Phi_k\,dv,
\end{equation}
where
$$
\Phi_k=\sum_{j=1}^kc_{kj}\derive{E_j}{N}\derive{E_{k+1-j}}{N}, \quad c_{kj}=\dfrac{k!}{(k+1-j)! j!}
$$
and $E_k(x)$ is the $k$-th exit time function defined in \eqref{esubk}. In particular, $\Omega$ is critical for the $k$-th exit time moment, for all $k\geq 1$ and for all volume preserving deformations, if and only if $\derive{E_k}{N}$ is constant on the boundary, for all $k$.
\end{lemme}

\begin{proof} For the proof of \eqref{fvetm} see Proposition 2.1 of \cite{mcD}, or Appendix 2. The last assertion follows from an obvious induction on $k$, observing that
$$
DT_1(\Omega,V)=-\int_{\bd\Omega}\scal{V}{N}\Big(\derive{E_1}{N}\Big)^2 d\sigma.
$$

\end{proof}


\subsection{Proof of Theorem \ref{etm}} 
Let $\mathcal H_0(\Omega)$ denote the vector space of all functions which are harmonic on $\Omega$ and have zero mean on the boundary:
$$
\mathcal H_0(\Omega)=\{\phi\in C^{\infty}(\Omega): \Delta\phi=0, \int_{\bd\Omega}\phi d\sigma=0\}.
$$

For $\phi\in\mathcal H_0(\Omega)$, we consider the function:
$$
H_{\phi}(t)=\int_{\Omega}\phi(x)u(t,x)\,dv(x),
$$
which is easily seen to represent the heat content of $\Omega$ with initial temperature distribution given by $\phi(x)$ and Dirichlet boundary conditions. 
It is proved in Theorem 9 of \cite{S1} that $\Omega$ has the constant flow property if and only if 
$H_{\phi}\equiv 0$ for all $\phi\in\mathcal H_0(\Omega)$.  An easy Fourier series  argument (see Lemma 8 of \cite{S1})  shows
 that the following upper bound holds for all $t\geq 0$:
$$
\abs{H_{\phi}(t)}\leq \abs{\Omega} \sup_{\Omega} \abs{\phi}\cdot e^{-\lambda_1 t},
$$
where $\lambda_1>0$ is the first Dirichlet eigenvalue of $\Omega$. Therefore, the Laplace transform
$$
\hat H_{\phi}(s)=\int_0^{\infty}e^{-st}H_{\phi}(t)\,dt
$$
is well-defined and analytic in the interval $(-\lambda_1,\infty)$. Moreover, as $H_{\phi}(t)$ is continuous on $[0,\infty)$, we have $H_{\phi}(t)=0$ for all $t$ if and only if $\hat H_{\phi}(s)=0$ for all $s$, by the injectivity of the Laplace transform. 

Let us compute the Taylor series of $\hat H_{\phi}(s)$ at $s=0$. One has 
$$
\begin{aligned}
\hat H_{\phi}(0)&=\int_0^{\infty}H_{\phi}(t)\,dt\\
&=\int_0^{\infty}\int_{\Omega}\phi(x)u(t,x)\,dv(x)dt\\
&=\int_{\Omega}\phi(x)\int_0^{\infty}u(t,x)\,dtdv(x)\\
&=\int_{\Omega}\phi(x)E_1(x)\,dv(x)
\end{aligned}
$$
by \eqref{expression}; more generally, the $k$-th derivative of $\hat H_{\phi}$ at $s=0$ is given by: 
\begin{equation}\label{hath}
\hat H_{\phi}^{(k)}(0)=\dfrac{(-1)^{k}}{k+1}\int_{\Omega}\phi E_{k+1}\,dv
\end{equation}
We will use \eqref{hath} to get the conclusion.
First assume that $\Omega$ has the constant flow property. Then $H_{\phi}\equiv 0$ for all 
$\phi\in\mathcal H_0(\Omega)$ (in particular $\int_{\Omega}\phi=0$) and, from \eqref{hath}:
$$
\int_{\Omega}\phi E_ndv=0
$$
for all $n\geq 0$. Integrating by parts we obtain:
\begin{equation}\label{byparts}
0=\int_{\Omega}\phi E_n\,dv=\dfrac{1}{n+1}\int_{\Omega}\phi\Delta E_{n+1}dv=
\dfrac{1}{n+1}
\int_{\bd\Omega}\phi\derive{E_{n+1}}Nd\sigma.
\end{equation}
for all $n\geq 0$. As any smooth function $\phi$ on $\bd\Omega$ admits an harmonic extension to the interior, we see that $\derive{E_{n+1}}N$ must be orthogonal to the subspace of zero mean functions on $\bd\Omega$, hence it must be  constant on $\bd\Omega$, for all $n\geq 0$. By Lemma \ref{eight} we conclude that $\Omega$ is critical for all $T_k$'s.

\smallskip

Conversely, assume that $\Omega$ is critical for all $T_k$'s. Then all normal derivatives $\derive{E_{k+1}}{N}$ are constant on the boundary of $\Omega$, so that  $\int_{\Omega}\phi E_n dv=0$ for all $n$ (from \eqref{byparts}).  From \eqref{hath}, we see that the Taylor series of $\hat H_{\phi}(s)$ at $s=0$ is zero; as $\hat H_{\phi}$ is analytic on $(-\lambda_1,\infty)$ we must have $\hat H_{\phi}\equiv 0$ hence also $H_{\phi}\equiv 0$. Then $\Omega$ has the constant flow property.



\section{Appendix 1: proof of Lemma \ref{rellich}}\label{appendixone}

In this section we consider a deformation $f_{\eps}:\Omega\to\Omega_{\eps}$ of the domain $\Omega$ along a vector field $V$ and identify $(\Omega_{\eps}, g)$ with $(\Omega,g_{\eps})$ where $g_{\eps}=f_{\eps}^{\star}g$.
We let $\Delta_{\eps}$ be the Laplacian associated to the metric $g_{\eps}$ of $\Omega$, and $dv_{\eps}$ the corresponding Riemannian measure.  Recall the notation:
$$
g'\doteq\dfrac{d}{d\eps}|_{\eps=0}g_{\eps},\quad \Delta'\doteq\dfrac{d}{d\eps}|_{\eps=0}\Delta_{\eps}, \quad dv'=\dfrac{d}{d\eps}|_{\eps=0}dv_{\eps}.
$$
We want to prove the following expression. 
\begin{lemme} \label{maintech} Assume that  the functions $\phi,\psi:\Omega\to\reals$ vanish on the boundary. Then:
$$
\begin{aligned}
\int_{\Omega}(\phi\Delta'\psi+\psi\Delta' \phi)dv=&
-2\int_{\Omega}\Big(\scal{V}{\nabla \phi}\Delta \psi+\scal{V}{\nabla \psi}\Delta \phi\Big)dv+
\int_{\Omega}(\psi\Delta \phi+\phi\Delta \psi)\delta Vdv\\
&+2\int_{\bd\Omega}\scal{V}{N}\derive \phi N\derive \psi Nd\sigma
\end{aligned}
$$
\end{lemme}

We first remark that it is enough to prove the lemma when $\phi=\psi$; the general formula will follow by polarization. That is, it is enough to show that, for all functions $\phi$ vanishing on the boundary:
\begin{equation}\label{short}
\int_{\Omega}\phi\Delta'\phi\, dv=
-2\int_{\Omega}\scal{V}{\nabla \phi}\Delta \phi\,dv+\int_{\Omega}\phi\Delta \phi\delta Vdv
+\int_{\bd\Omega}\scal{V}{N}\Big(\derive \phi N\Big)^2\,d\sigma
\end{equation}

In what follows, we denote by $\delta$ the adjoint of the gradient operator, so that 
$
\delta V=-{\rm div V}=-\sum_{j=1}^n\scal{\nabla_{e_j}V}{e_j}
$
where $(e_1,\dots, e_n)$ is any orthonormal frame. We also recall that if  $h$ is a symmetric two-tensor then $\delta h$ is by definition the one-form:
$$
\delta h(X)=-\sum_{j=1}^n\nabla_{e_j}h(e_j,X).
$$

\begin{lemme} \label{prime} Let $g', \Delta'$ and $dv'$ be as above. Let $V$ be the deformation vector field. Then $g'=\mathcal L_Vg$ and one has:
\begin{equation}\label{lemmauno}
g'(X,Y)=g(\nabla_XV,Y)+g(X,\nabla_YV)=\nabla V(X,Y)+\nabla V(Y,X).
\end{equation}
\begin{equation}\label{lemmadue}
\scal{g'}{g}=-2\delta V
\end{equation}
\begin{equation}\label{lemmatre}
dv'=-\delta Vdv
\end{equation}
\begin{equation}\label{deltaprime}
\Delta'u=\scal{\nabla^2 u}{g'}-\scal{du}{\delta g'}-\dfrac 12\scal{du}{d\scal{g}{g'}}
\end{equation}
\end{lemme}

In fact, formulae \eqref{lemmauno}, \eqref{lemmadue}, \eqref{lemmatre} are standard, while \eqref{deltaprime} is due to Berger \cite{B}.

Lemma \eqref{maintech} will drop from the following two facts.

\begin{lemme}\label{maintechone} One has for every $\phi$ vanishing on the boundary:
$$
\int_{\Omega}\phi\Delta'\phi\, dv=-2\int_{\Omega}\nabla V(\nabla \phi,\nabla \phi)\,dv+\frac 12\int_{\Omega}(\Delta \phi^2)\delta V\,dv
$$
\end{lemme}

The next one is a Rellich-type identity.

\begin{lemme}\label{maintechtwo} Let $X$ be any vector field on $\Omega$ and assume that $\phi$ vanishes on $\bd\Omega$. Then:
$$
2\int_{\Omega}\nabla X(\nabla \phi,\nabla \phi)\,dv=
\int_{\Omega}\Big(2\scal{X}{\nabla \phi}\Delta \phi-\abs{\nabla\phi}^2\delta X\Big)dv
-\int_{\bd\Omega}\scal{X}{N}\Big(\derive \phi N\Big)^2\,d\sigma
$$
\end{lemme}
Lemma \ref{maintechone} and Lemma \ref{maintechtwo} will be proved below. 

\subsection{Proof of \eqref{short}}
It follows by taking $X=V$ in Lemma \ref{maintechtwo} and by inserting the resulting identity in Lemma \ref{maintechone}.


\subsection{Proof of Lemma \ref{maintechone}} 

From \eqref{deltaprime} we see that:
$$
\int_{\Omega}\phi\Delta'\phi\,dv=\int_{\Omega}\Big(\scal{\phi\nabla^2 \phi}{g'}-\scal{\phi d\phi}{\delta g'}-\dfrac 12\scal{\phi d\phi}{d\scal{g}{g'}}\Big)\,dv
$$
Now, by Green formula and \eqref{lemmadue}:
$$
\begin{aligned}
-\int_{\Omega}\dfrac 12\scal{\phi d\phi}{d\scal{g}{g'}}\,dv&=\int_{\Omega}\scal{\phi d\phi}{d\delta V}\,dv\\
&=\int_{\Omega}\delta(\phi d\phi)\delta V\,dv\\
&=\int_{\Omega}\Big(\phi\Delta \phi-\abs{d\phi}^2\Big)\delta V\,dv\\
&=\frac 12\int_{\Omega}\Delta\phi^2\delta V\,dv\\
\end{aligned}
$$
Hence:
\begin{equation}\label{uno}
\begin{aligned}
\int_{\Omega}\phi\Delta'\phi\,dv=\int_{\Omega}\scal{\phi\nabla^2 \phi}{g'}\,dv-
\frac 12 \int_{\Omega}\scal{d\phi^2}{\delta g'}dv+
\frac 12\int_{\Omega}\Delta\phi^2\delta V\,dv\\
\end{aligned}
\end{equation}
Taking into account that
\begin{equation}
\nabla^2(\phi^2)=2\phi\nabla^2\phi+2d\phi\otimes d\phi
\end{equation}
and that, by Green formula and the fact that $\phi$ vanishes on the boundary:
\begin{equation}
\int_{\Omega}\scal{\nabla^2(\phi^2)}{g'}dv=\int_{\Omega}\scal{d\phi^2}{\delta g'}dv
\end{equation}
we see that:
\begin{equation}\label{due}
\int_{\Omega}\scal{\phi\nabla^2\phi}{g'}dv-\frac 12 \int_{\Omega}\scal{d\phi^2}{\delta g'}dv=-\int_{\Omega}\scal{d\phi\otimes d\phi}{g'}dv
\end{equation}
and then, substituting \eqref{due} in \eqref{uno}:
$$
\int_{\Omega}\phi\Delta'\phi\,dv=-\int_{\Omega}\scal{d\phi\otimes d\phi}{g'}dv+\frac 12\int_{\Omega}\Delta\phi^2\delta V\,dv.
$$
We now observe that
$$
\scal{d\phi\otimes d\phi}{g'}=g'(\nabla \phi,\nabla \phi)=2\nabla V(\nabla \phi,\nabla \phi)
$$
and we arrive at the final expression.


\subsection{Proof of Lemma \ref{maintechtwo}}

For the proof, we start from the pointwise identity:
\begin{equation}\label{identity}
\delta\Big(2\scal{X}{\nabla \phi}\nabla \phi-\abs{\nabla\phi}^2X\Big)=-2\nabla X(\nabla \phi,\nabla \phi)
+2\scal{X}{\nabla \phi}\Delta \phi-\abs{\nabla\phi}^2\delta X
\end{equation}
To see it, first observe that
\begin{equation}\label{primo}
\begin{aligned}
\delta(2\scal{X}{\nabla\phi}\nabla\phi)&=-2\nabla\phi\cdot\scal{X}{\nabla\phi}+2\scal{X}{\nabla\phi}\Delta\phi\\
&=-2\scal{\nabla_{\nabla\phi}X}{\nabla\phi}
-2\scal{X}{\nabla_{\nabla\phi}\nabla\phi}+2\scal{X}{\nabla\phi}\Delta\phi\\
&=-2\nabla X(\nabla\phi,\nabla\phi)-2\nabla^2\phi(\nabla\phi,X)+2\scal{X}{\nabla\phi}\Delta\phi
\end{aligned}
\end{equation}

and that:
\begin{equation}\label{secondo}
\begin{aligned}
\delta(-\abs{\nabla\phi}^2X)&=2\scal{\nabla_X\nabla\phi}{\nabla\phi}-\abs{\nabla\phi}^2\delta X\\
&=2\nabla^2\phi(X,\nabla\phi)-\abs{\nabla\phi}^2\delta X
\end{aligned}
\end{equation}
so \eqref{identity} follows immediately by adding \eqref{primo} and \eqref{secondo}
(recall that $\nabla^2\phi$ is symmetric).

Lemma \ref{maintechtwo} follows integrating \eqref{identity}  over $\Omega$ and using the Green formula, also recalling that, since $\phi$  vanish on the boundary, we have
$
\nabla\phi=\derive{\phi}{N}N
$
on $\bd\Omega$.

\section{Appendix 2: proof of Lemma \ref{eight}} We want to prove the following fact.

\begin{lemme}\label{eight} The first variation of the $k$-th exit time moment $T_k$, at $\Omega$, in the direction $V$, is given by:
\begin{equation}\label{fvetm}
DT_k(\Omega,V)=-\int_{\bd\Omega}\scal{V}{N}\Phi_k\,d\sigma,
\end{equation}
where
$$
\Phi_k=\sum_{j=1}^kc_{kj}\derive{E_j}{N}\derive{E_{k+1-j}}{N}, \quad c_{kj}=\dfrac{k!}{(k+1-j)! j!}
$$
and $E_k(x)$ is the $k$-th exit time function defined in \eqref{esubk}. In particular, $\Omega$ is critical for the $k$-th exit time moment, for all $k\geq 1$ and for all volume preserving deformations, if and only if $\derive{E_k}{N}$ is constant on the boundary, for all $k$.
\end{lemme}

It is enough to show \eqref{fvetm}, as the last assertion will follow from an obvious inductive argument.

\smallskip

It will be simpler to work with the functions $u_k=\frac{1}{k!}E_k$ which satisfy 
$u_0=1$ and
$$
\twosystem
{\Delta u_k=u_{k-1}\qtq\Omega}
{u_k=0\qtq\bd\Omega}
$$
for $k\geq 1$. Proving \eqref{fvetm} is in turn equivalent to proving the following fact. Let $Q_k(\Omega)$ be the functional
\begin{equation}\label{intform}
Q_k(\Omega)=\int_{\Omega} u_k\,dv.
\end{equation}
Then:
\begin{equation}\label{nduk}
Q'_k(\Omega)\doteq DQ_k(\Omega,V)=-\int_{\bd\Omega}\scal{V}{N}\Psi_k\,d\sigma, \quad\text{where}\quad \Psi_k=\sum_{j=1}^k\derive{u_j}{N}\derive{u_{k+1-j}}{N}.
\end{equation}
We then proceed to show \eqref{nduk}.

\smallskip

{\bf Step 1.} We first prove:
\begin{equation}\label{primopasso}
DQ_k(\Omega,V)=-\frac 12\sum_{j=1}^{k}\int_{\Omega}\Big(u_j\Delta' u_{k+1-j}+u_{k+1-j}\Delta' u_j\Big)\,dv+\int_{\Omega}u_kdv'.
\end{equation}
For the proof, we start differentiating \eqref{intform}:
\begin{equation}\label{diff}
Q_k'=\int_{\Omega}u_k'\,dv+\int_{\Omega}u_k\,dv'.
\end{equation}
We take care of the first term. One has:
$$
\begin{aligned}
\int_{\Omega}u_k'\,dv&=\int_{\Omega}u_k'\Delta u_1\,dv\\
&=\int_{\Omega}u_1\Delta u_k'\,dv
\end{aligned}
$$
Differentiate the identity $\Delta u_k=u_{k-1}$ and obtain:
$$
\Delta' u_k+\Delta u_k'=u_{k-1}'.
$$
Recalling that $u_1$ and $u_k$ vanish on the boundary, we see by Green formula and the above:
\begin{equation}\label{fit}
\int_{\Omega}u_k'\,dv=-\int_{\Omega}u_1\Delta' u_k\,dv+\int_{\Omega}u_1u'_{k-1}\,dv.
\end{equation}
This is the first iteration. We now take care of the second integral, in a similar way:
$$
\begin{aligned}
\int_{\Omega}u_1u'_{k-1}\,dv&=\int_{\Omega}\Delta u_2\cdot u'_{k-1}\,dv\\
&=\int_{\Omega}u_2\Delta u'_{k-1}\,dv\\
&=-\int_{\Omega}u_2\Delta' u_{k-1}dv+\int_{\Omega}u_2u'_{k-2}\,dv.
\end{aligned}
$$
After two iterations we then have:
$$
\int_{\Omega}u_k'\,dv=-\int_{\Omega}\Big(u_1\Delta' u_k+u_2\Delta'u_{k-1}\Big)\,dv+\int_{\Omega}u_2u'_{k-2}\,dv.
$$
It is clear how to continue. The process stops at step $k$, because $u'_0=0$. Finally we get:
\begin{equation}\label{stepk}
\begin{aligned}
\int_{\Omega}u_k'\,dv&=-\sum_{j=1}^{k}\int_{\Omega}
u_j\Delta' u_{k+1-j}\,dv\\
&=-\frac 12\sum_{j=1}^{k}\int_{\Omega}
\Big(u_j\Delta' u_{k+1-j}+u_{k+1-j}\Delta'u_j\Big)\,dv
\end{aligned}
\end{equation}
Then \eqref{primopasso} follows from \eqref{stepk} and \eqref{diff}.  

\smallskip

{\bf Step 2.} Taking into account \eqref{primopasso}, to prove the final statement it is enough to show that
\begin{equation}\label{showit}
\begin{aligned}
-\sum_{j=1}^k&\int_{\Omega}(u_j\Delta' u_{k+1-j}+u_{k+1-j}\Delta' u_{j})\,dv\\
&=2\int_{\Omega}u_k\delta V\,dv-2\sum_{j=1}^k\int_{\bd\Omega}\scal{V}{N}\derive{u_j}{N}\derive{u_{k+1-j}}{N}\,d\sigma.
\end{aligned}
\end{equation}
Once this identity has been established, formula \eqref{nduk} follows because
$
\int_{\Omega}u_kdv'_g=-\int_{\Omega}u_k\delta V\,dv
$
From Lemma \ref{maintech} we see that:
$$
\begin{aligned}
-\int_{\Omega}&(u_j\Delta' u_{k+1-j}+u_{k+1-j}\Delta' u_{j})\,dv\\
&=\int_{\Omega}\Big(2\scal{V}{\nabla u_j}\Delta u_{k+1-j}+
2\scal{V}{\nabla u_{k+1-j}}\Delta u_{j}\Big)\,dv\quad\label{A(j)}\\
&-\int_{\Omega}\Big(u_j\Delta u_{k+1-j}+u_{k+1-j}\Delta u_j\Big)\delta V\,dv\quad\label{B(j)}\\
&-2\int_{\bd\Omega}\scal{V}{N}\derive{u_j}{N}\derive{u_{k+1-j}}{N}\,d\sigma.
\end{aligned}
$$
The sum of the two inner integrals, thanks to the defining relations, and after rearranging terms, equals:
$$
\begin{aligned}
A(j)+B(j)&=\int_{\Omega}\Big(2\scal{V}{\nabla {u_j}}u_{k-j}-u_ju_{k-j}\delta V\Big)\,dv\quad\label{C(j)}\\
&+\int_{\Omega}\Big(2\scal{V}{\nabla {u_{k+1-j}}}u_{j-1}-u_{k+1-j}u_{j-1}\delta V\Big)\,dv\quad\label{D(j)}
\end{aligned}
$$
We now sum over $j=1,\dots,k$. One sees that if $\psi=\sum_{j=1}^{k-1}u_ju_{k-j}$, then 
$$
\sum_{j=1}^{k-1}C(j)=\int_{\Omega}\Big(\scal{V}{\nabla\psi}-\psi\delta V\Big)dv=0
$$
by the Green formula, since $\psi$ vanishes on the boundary.
Hence the only term surviving in the sum is when $j=k$, which implies that:
$$
\sum_{j=1}^kC(j)=C(k)=\int_{\Omega}\Big(2\scal{V}{\nabla u_k}-u_k\delta V\Big)\,dv=\int_{\Omega}u_k\delta V\,dv.
$$
Similarly one sees that $\sum_{j=2}^kD(j)=0$, the only term surviving in the second sum is when $j=1$, and then:
$$
\sum_{j=1}^kD(j)=D(1)=\int_{\Omega}\Big(2\scal{V}{\nabla u_k}-u_k\delta V\Big)\,dv=\int_{\Omega}u_k\delta V\,dv.
$$
The last two identities imply that 
$$
\sum_{j=1}^k(A(j)+B(j))=\sum_{j=1}^kC(j)+\sum_{j=1}^kD(j)=2\int_{\Omega}u_k\delta V\,dv,
$$
and \eqref{showit} holds. 


\section{Appendix 3: proof of Theorem \ref{inthesphere}}
We first make some preliminary considerations and prove a lemma. Recall that every isoparametric hypersurface is a regular level set of the restriction to $\sphere n$ 
of a  Cartan polynomial in $\real{n+1}$ (as defined in \eqref{cmp}): denote such restriction by $V$. Then, $V$ satisfies (see for example \cite{Shk}):
\begin{equation}\label{CM}
\twosystem
{\abs{\nabla V}^2=g^2(1-V^2)}
{\Delta V=g(g+n-1)V-c}
\end{equation}
where $\nabla$ and $\Delta$ denote the gradient and the Laplacian of $\sphere n$. Now $V$ takes values in $[-1,1]$; every level set $V^{-1}(t)$ is a (regular) isoparametric hypersurface for $t\in (-1,1)$ while $\Sigma_+=V^{-1}(1)$ and 
$\Sigma_-=V^{-1}(-1)$ are the focal submanifolds. 
\smallskip

We now focus on the focal submanifold $\Sigma_+=f^{-1}(1)$, and let 
$\rho$ be the distance function to $\Sigma_+$:
$$
\rho(x)=d(x, \Sigma_+).
$$
$V$ and $\rho$ are related  by the formula:
\begin{equation}\label{vrho}
V(x)=\cos(g\rho(x)),
\end{equation}
so that $\rho$ maps onto $[0,\frac{\pi}g]$; note that $\rho$ is smooth on the complement of the focal set.  It is well-known that, if $x$ is a regular point of $\rho$, then  $\Delta\rho(x)$ is the mean curvature (i.e. trace of the second fundamental form) of the equidistant $\rho^{-1}(\rho(x))$ through $x$, with respect to the unit normal vector $\nabla\rho(x)$. 
We have the following fact.

\begin{lemme} At all regular points of $\rho$ (hence, on the complement of the focal set) one has:
$$
\Delta\rho=-(n-1)\cot(g\rho)+
\frac{c}{g\sin(g\rho)},
$$
where $c$ is as in \eqref{defc}.
\end{lemme}

\begin{proof} Equation \eqref{vrho} says that $V=\psi\circ\rho$ where $\psi(r)=\cos(gr)$. Knowing that $\abs{\nabla\rho}=1$, we see:
$$
\Delta V=-\psi''\circ\rho+(\psi'\circ\rho)\Delta\rho=g^2\cos(g\rho)-g\sin(g\rho)\Delta\rho.
$$
Hence, by \eqref{CM}:
$$
g^2\cos(g\rho)-g\sin(g\rho)\Delta\rho=g(g+n-1)\cos(g\rho)-c,
$$
which gives the assertion.
\end{proof}
We can now prove the theorem:

\begin{thm} Let $\Omega$ be a domain in $\sphere n$. Then $\Omega$ is an isoparametric tube if and only if :

\smallskip

a) either $\Omega$ is bounded by a connected isoparametric hypersurface,

\smallskip

b) or $\Omega$ is a tube around a minimal isoparametric hypersurface $\Sigma$ such that all its distinct principal curvatures have the same multiplicity (that is, $\Sigma$ is minimal with $c=0$).
\end{thm} 

\begin{proof} First, assume that $\Omega$ is an isoparametric tube: we know that then $\bd\Omega$ can have at most two components, each being an isoparametric hypersurface.  

\smallskip

If the boundary is connected we have a). 

\smallskip

If the boundary has two components then the soul has to be a minimal isoparametric hypersurface $\Sigma$ and we  just need to show that $c=0$. 

\smallskip

We fix the focal submanifold $\Sigma_+$ of the foliation, as above, and let $\rho$ be the distance function to $\Sigma_+$. The mean curvature of $\rho^{-1}(r)$, with respect to the normal vector $\nabla\rho$, is radial, that is, it depends only on $r$; by the lemma, it is written $\Delta\rho=\eta\circ\rho$ where 
\begin{equation}\label{eta}
\eta(r)=-(n-1)\cot(gr)+\frac{c}{g\sin(gr)}.
\end{equation}
The soul $\Sigma$ of $\Omega$ is the unique minimal member of the foliation; by \eqref{eta}, it is at distance $r_0$ to $\Sigma_+$, with $r_0$ satisfying $\eta(r_0)=0$, that is:
$$
\cos(gr_0)=\dfrac{c}{(n-1)g}, \quad\text{hence} \quad r_0=\frac 1g\arccos\Big(\frac{c}{(n-1)g}\Big).
$$
By assumption $\Omega$ is an isoparametric tube around $\Sigma$: if $d$ denotes the distance function to $\Sigma$, then the equidistant set $d^{-1}(r)$ has two components $\Gamma_1(r)$ and $\Gamma_2(r)$. 
Now the component which is closer to $\Sigma_+$, say $\Gamma_1(r)$, is at constant distance $r_0-r$ to $\Sigma_+$ while the other is at distance $r_0+r$. 
It is also clear that, if we choose the inner unit normal to $\bd\Omega$, then the mean curvature of $\Gamma_1(r)$ is $\eta(r_0-r)$ while the mean curvature of $\Gamma_2(r)$ is $-\eta(r_0+r)$. In order for $\Omega$ to be an isoparametric tube, the two components must have the same mean curvature, hence:
$$
\eta(r_0-r)=-\eta(r_0+r)
$$
for all $r$ in the appropriate range. Looking at the formula \eqref{eta} we see easily that $\eta(r)$ is odd with respect to $r=r_0$ if and only if $c=0$, in which case $r_0=\frac{\pi}{2g}$. 

\smallskip

The conclusion is that if $\Omega$ is an isoparametric tube with two boundary components then b) holds.  

\smallskip

Conversely, if either a) or b) hold, the $\Omega$ is shown to be an isoparametric tube by the same arguments as above. The proof is complete. 

\end{proof}

\addcontentsline{toc}{chapter}{Bibliography}
\bibliographystyle{plain}
\bibliography{heatcontent}

\bigskip
\noindent
\small
Alessandro Savo \\
 Dipartimento SBAI, Sezione di Matematica,
Sapienza Universit\`a di Roma\\
Via Antonio Scarpa 16\\
00161 Roma, Italy

\noindent alessandro.savo@uniroma1.it

\end{document}